\RequirePackage{fix-cm}%
\documentclass[a4paper,oneside,reqno]{amsart}
\usepackage[latin1]{inputenc}
\usepackage[T2A,T1]{fontenc}
\usepackage{lmodern,fixltx2e}%
\def\SetBibliographyCyrillicFontfamily{\fontfamily{cmr}}%
\def\Cyr#1{\bgroup\SetBibliographyCyrillicFontfamily
                  \fontencoding{T2A}\selectfont{#1}\egroup}
\usepackage{etex}%
\usepackage[german,british]{babel}
\usepackage{hyphxcpt}
\usepackage[a4paper,verbose,textwidth=338.0pt]{geometry}
\usepackage{tikz}%
\newif\ifforceblacklinks\forceblacklinksfalse

\def\ZhukFullName{\Cyr{\CYRD\cyrm\cyri\cyrt\cyrr\cyri\cyrishrt{}
                       \CYRZH\cyru\cyrk{}}}%
\def\PublicationTitle{The number of clones determined by
                      disjunctions of unary relations}%
\def\CorrespondingAuthor{Mike Behrisch}%
\def\SecondAuthor{Edith Vargas-Garc\'{i}a}%
\def\ThirdAuthor{\ZhukFullName{} (Dmitriy Zhuk)}%
\newdimen{\oldfontdimen}%
\def\TUWname{\foreignlanguage{german}{%
             \oldfontdimen=\fontdimen2\font%
             \bgroup%
             \fontdimen2\font=0.97\fontdimen2\font
             Tech\-ni\-sche Uni\-ver\-si\-t\"{a}t Wi\nlb{}en%
             \egroup
             \fontdimen2\font=\oldfontdimen%
             }}%
\def\InstitutDMG{\foreignlanguage{german}{%
             In\-sti\-tut f\"{u}r Dis\-kre\-te Ma\-the\-ma\-tik
             und Geo\-me\-trie}}%
\def\PostleitzahlWien{\mbox{A-1040} Vienna}%
\def\ULeeds{ITAM}%
\def\SchoolName{Departamento Acad\'{e}mico de Matem\'aticas}%
\def\EdithAddress{R\'{\i}o Hondo No.~1,
                  Col. Tizap\'{a}n San \'{A}ngel,
                  Del. \'{A}lvaro Obreg\'{o}n,
                  C.P.~01080 M\'{e}xico, D.F.
                  }%
\def\MoscowState{Lomonosov Moscow State University}%
\def\DmitrijDept{Department of Mathematics and Mechanics}%
\def\DmitrijSubDept{Mathematical Theory of Intelligent Systems}%
\def\DmitrijAddress{Vorobjovy Hills,
                    Moscow,
                    Russian Federation,
                    119899}%
\def\PublicationTopic{Cardinality of the lattice of clones determined
                      by disjunctively definable relations over unary
                      predicates}%
\def\PublicationKeywords{clone,
                         disjunctive definition,
                         unary relation,
                         cross,
                         clausal constraint,
                         signed logic%
                         }
\usepackage{amsmath}
\usepackage{amssymb}
\usepackage{amsfonts}
\usepackage{amsthm}
\theoremstyle{plain}
  \newtheorem{theorem}{Theorem}
  \newtheorem{lemma}[theorem]{Lemma}
  \newtheorem{proposition}[theorem]{Proposition}
  \newtheorem{corollary}[theorem]{Corollary}

\theoremstyle{definition}
  \newtheorem{definition}[theorem]{Definition}
  
  \newtheorem{remark}[theorem]{Remark}

\usepackage{enumerate}
\usepackage{graphicx}
\usepackage{url}
\usepackage{ifpdf}
\usepackage{ifthen}%
\usepackage{tudcolors}%
\usepackage{xspace}%
\usepackage{dsfont}%
\usepackage{mleftright}%

\ifpdf% pdfTEX is running in pdf mode
  \PassOptionsToPackage{pdftex,breaklinks}{hyperref}%
\fi%
\usepackage{hyperref}%

\newcommand{\m}[1]{\mbox{\ensuremath{#1}}}% a command for boxed math mode.
\DeclareRobustCommand{\nbdd}[1]{\m{#1}\nobreakdash-\hspace{0pt}}% from amsmath, nonbreaking dash + math mode
\DeclareRobustCommand{\dash}{\nolinebreak\hspace{0pt}-\hspace{0pt}}% a dash
\DeclareRobustCommand{\nlb}{\penalty10000\hskip0pt\relax}%
\newcommand{\ie}{i.e.\xspace}
\newcommand{\eg}{e.g.\xspace}
\newcommand{\wrt}{w.r.t.\xspace}
\newcommand{\N}{\mathds{N}}
\newcommand{\Np}{\mathds{N}_{+}}
\newcommand{\powerset}[1]{\ensuremath{\mathfrak{P} \apply{#1}}}
\newcommand{\subs}{\subseteq}
\newcommand{\preserves}{\ensuremath{\rhd}}
\newcommand{\defeq}{\mathrel{\mathop:}=}
\DeclareRobustCommand{\lset}[2]{\ensuremath{\mleft\{\mleft.#1\ \vphantom{#2}\mright|\ #2\mright\}}}
\newcommand{\set}[1]{\ensuremath{\mleft\{ #1 \mright\}}}
\newcommand{\apply}[1]{\ensuremath{\mleft( #1 \mright)}}
\newcommand{\fapply}[1]{\ensuremath{\mleft[ #1 \mright]}}
\newcommand*{\abs}[1]{\ensuremath{\mleft|#1\mright|}}
\newcommand{\factorBy}[2]{\ensuremath{#1 / #2}}
\newcommand{\functionhead}[3]{\ensuremath{#1\colon #2 \longrightarrow #3}}
\newcommand{\function}[5]{\ensuremath{\begin{array}{cccc}
#1\colon & #2 & \longrightarrow & #3 \\
    & #4 & \longmapsto     & \displaystyle #5
\end{array}}}
\newcommand{\liste}[2]{\ensuremath{#1_{1},\ldots,#1_{#2}}}
\newcommand{\listen}[2]{\ensuremath{#1_{0},\ldots,#1_{#2 - 1}}}

\renewcommand{\phi}{\varphi}
\renewcommand{\rho}{\varrho}

\DeclareMathOperator{\id}{id}
\DeclareMathOperator{\im}{im}

\newcommand{\name}[1]{\textsc{#1}}
\DeclareMathOperator{\Cr}{Cross}%
\DeclareMathOperator{\pr}{pr}%
\DeclareMathOperator{\emb}{emb}%
\DeclareMathOperator{\Pfin}{\mathfrak{P}_{\mathrm{fin}}}%
\newcommand{\powersetfin}[1]{\ensuremath{\Pfin\apply{#1}}}%
\newcommand{\inV}{\prime}%
\newcommand{\poL}{\backprime}%
\newcommand{\CarrierSet}[1][]{% first arg is optional (power)
  \ifthenelse{\equal{#1}{}}{%
    \ensuremath{A}
  }{%
    \ensuremath{A^{#1}}
  }
}
\DeclareMathOperator{\PolOp}{Pol}%
\newcommand{\Pol}[1]{\PolOp_{A} #1}%
\DeclareMathOperator{\OpOp}{O}%
\newcommand{\Ops}{\OpOp_A}%
\newcommand{\composition}[2]{\ensuremath{#1\circ\mleft(#2\mright)}}
\newcommand{\eni}[1]{\ensuremath{e^{(n)}_{#1}}}%
\DeclareMathOperator{\DDOp}{DD}
\newcommand{\DD}[1][\Gamma]{\ensuremath{\DDOp\apply{#1}}}%
\DeclareMathOperator{\ptOp}{pt}
\newcommand{\pt}[1]{\ensuremath{\ptOp\apply{#1}}}%
\DeclareMathOperator{\CRelOp}{R}
\newcommand{\Runary}[1]{\ensuremath{\CRelOp\apply{#1}}}%
\newcommand{\below}{\sqsubseteq}%
\newcommand{\ordrel}{\leq}%
\newcommand{\lt}{\ordrel}%
\newcommand{\genericOrder}{P}%
\newcommand{\order}[1][\genericOrder]{\ensuremath{\mathbb{#1}}}%
\newcommand{\ordset}[3]{\ensuremath{\order[{#1}] = \apply{#2, #3}}}%
\newcommand{\poset}[1][{\genericOrder}]{\ordset{#1}{#1}{\ordrel}}%
\newcommand{\downset}[2][{\order}]{\ensuremath{%
                        \mathop{\downarrow_{#1}}\mathopen{}#2}}%
\DeclareMathOperator{\DSOp}{DS}%
\newcommand{\DS}[1][{\order}]{\DSOp\apply{#1}}%
\newcommand{\upset}[2][{\order}]{\ensuremath{%
                        \mathop{\uparrow_{#1}}\mathopen{}#2}}%
\DeclareMathOperator{\USOp}{US}%
\newcommand{\US}[1][{\order}]{\USOp\apply{#1}}%
\newcommand{\restrordrel}[2][\lt]{\ensuremath{%
                         \mathrel{#1\mathclose\restriction_{#2}}}}%
\DeclareMathOperator{\suppOP}{supp}%
\newcommand{\supp}[1]{\suppOP\apply{#1}}%
\newcommand{\YF}[1][Y]{\ensuremath{#1_{\mathrm{F}}}}%
\newcommand{\Ysubs}[2][Y]{\ensuremath{#1_{\subs \mathnormal{#2}}}}%
\hypersetup{
    bookmarks=true,          % show bookmarks bar?
    unicode=true,            % non-Latin characters in Acrobat's bookmarks
    pdftoolbar=true,         % show Acrobat's toolbar?
    pdfmenubar=true,         % show Acrobat's menu?
    pdffitwindow=false,      % window fit to page when opened
    pdfstartview={FitH},     % fits the width of the page to the window
    pdfdisplaydoctitle=true, % display document title instead of file name in title bar
    pdfnewwindow=true,       % links in new window
    linktocpage=true,        % in toc only pages are linked, not the text
    colorlinks=true,         % false: boxed links; true: colored links
    linkcolor=HKS36K100,     % color of internal links
    citecolor=HKS41K100,     % color of links to bibliography
    filecolor=HKS33K100,     % color of file links
    urlcolor=HKS44K100,      % color of external links
    pdftitle={\PublicationTitle},    % title
    pdfauthor={\CorrespondingAuthor},     % author
    pdfsubject={Notes: \PublicationTopic},   % subject of the document
    pdfkeywords={\PublicationKeywords},         % list of keywords
}
\ifforceblacklinks%
\hypersetup{%
    linkcolor=black,     % color of internal links
    citecolor=black,     % color of links to bibliography
    filecolor=black,     % color of file links
    urlcolor=black       % color of external links
}%
\fi%
\numberwithin{equation}{section} % local equation numbering
\allowdisplaybreaks

\begin{document}

\thispagestyle{empty}
\selectlanguage{british}

\title{\PublicationTitle}
\author[M. Behrisch]{\CorrespondingAuthor}%
\address{\InstitutDMG\\
         \TUWname\\
         \PostleitzahlWien\\
         Austria}%
\email{behrisch@logic.at}%
\thanks{This is a manuscript version of an article to appear in
        Theory of Computing Systems.
        A restricted version of this result was presented at the
        89.~Ar\-beits\-ta\-gung All\-ge\-mei\-ne Al\-ge\-bra
        (89th Workshop on General Algebra), AAA89, that took place in
        Dres\-den, Germany, from 27 February to 1 March 2015.}
\thanks{The research of the first author was partly support by the
        Austrian Science Fund (FWF) under grant I836-N23, and by
        the OeAD KONTAKT project CZ 04/2017 ``Ordered structures for
        non-classical logics''. The second author acknowledges partial
        support by the Asociaci\'{o}n Mexicana de Cultura A.C}
\author[E. Vargas-Garc\'{i}a]{\SecondAuthor}%
\address{\SchoolName\\
         \ULeeds\\
         \EdithAddress}%
\email{edith.vargas@itam.mx}%

\author[D. Zhuk]{\ThirdAuthor}%
\address{\DmitrijSubDept\\
         \DmitrijDept\\
         \MoscowState\\
         \DmitrijAddress}%
\urladdr{http://intsys.msu.ru/en/staff/zhuk/}%
\email{zhuk.dmitriy@gmail.com}%
\date{\today}
%% Keywords and phrases
%% AMS subject classification; see http://www.ams.org/msc
%% Only one Primary. Possibly several Secondary.
\subjclass[2010]{Primary:
  08A40; % General algebraic systems, Algebraic structures,
         % Operations, polynomials, primal algebras
  Secondary:
  08A02, % General algebraic systems, Algebraic structures,
         % Relational systems, laws of composition
  08A99. % General algebraic systems, Algebraic structures,
         % None of the above, but in this section
  }%

\begin{abstract}
We consider finitary relations (also known as crosses) that are definable via finite
disjunctions of unary relations, \ie\ subsets, taken from a fixed finite
parameter set~\m{\Gamma}. We prove that whenever~\m{\Gamma} contains at
least one non\dash{}empty relation distinct from the full carrier set,
there is a countably infinite number of polymorphism clones determined
by relations that are disjunctively definable from~\m{\Gamma}. Finally, we
extend our result to finitely related polymorphism clones and countably
infinite sets~\m{\Gamma}.
\end{abstract}
\maketitle

\section{Introduction}\label{sect:intro}
Constraint Satisfaction Problems (CSPs) offer a uniform framework to
study algorithmic problems. In one of the simplest forms one is given a
conjunctive formula over some chosen parameter set of finitary
relations~\m{Q} and is asked to decide whether the formula is satisfiable.
Even in this basic manifestation many important problems can be encoded
as CSPs, for instance, graph \nbdd{k}colourability, unrestricted Boolean
satisfiability
(SAT), Boolean \nbdd{3}satisfiability (\nbdd{3}SAT) and further variants of
this problem, solvability of sudokus, the \nbdd{n}queens problem, more
generally the exact cover problem, and many others.
\par
There is an active stream of research producing results regarding the
decision complexity of CSPs and its variants, using various
approaches~\cite{%
HellNesetril_CSPDichotomyUndirectedGraphs,%
SchnoorSchnoor_PartialPolymorphismsandCSP,%
BulatovValerioteResultsOnTheAlgebraicApproach2CSP,%
MarotiMcKenzieExistenceThmsForWeaklySymmetricOperations,%
LaroseTessonUniversalAlgebraAndHardnessResultsForCSP,%
BermanIdziakMarkovicMcKenzieValerioteWillardVarietiesWithFewSubalgebrasOfPowers,%
BartoKozikAbsorbingSubalgebrasCyclicTermsCSP,%
BartoKozikWillard_NUConstraintsBoundedPathwidthDuality,%
Barto_CSPandUA,%
BartoKozik_CSPLocalConsistencyMethods,%
BulinDelicJacksonNiven_FinerReductionCSPDigraphs,%
BartoPinsker_infiniteCSP,%
MayrChen_QCSPonMonoids,%
ChenValerioteYuichiTestingAssignments2CSPs,%
BartoKozik_RobustlySolvableCSPs,%
DapicMarkovicMartin_QCSPSemicompleteDigraphs,%
BartoKozik_AbsorptionInUA,%
BartoKrokhinWillard_PolymorphismsHowTo%,
}.
Moreover, recently there have been credible claims regarding the
solution~\cite{BulatovCSPproofArchive,ZhukCSPProofArchive,ZhukISMVL2017CSPproof}
of Feder and Vardi's famous CSP dichotomy
conjecture~\cite{FederVardiCSP}
stating that any CSP on a finite set can either be decided in polynomial
time (is \emph{tractable}) or is NP-complete, otherwise.
\par

A more algebraic formulation of a CSP is given by fixing a relational
structure~\m{\mathbb{A}} of finite signature on a finite set~\m{A} with
set of basic relations~\m{Q}. The question to decide is, for any
finite relational structure~\m{\mathbb{B}} of the same signature
as~\m{\mathbb{A}}, whether there is a homomorphism from~\m{\mathbb{B}}
to~\m{\mathbb{A}}. A basic reduction result attributed to
Jeavons~\cite{JeavonsAlgebraicStructureOfCSP} implies that any two
CSPs on the same carrier set~\m{\CarrierSet} parametrized by finite sets
of relations \m{Q_1} and \m{Q_2} sharing the same polymorphism clone
(see Section~\ref{sect:prelim}) have the same complexity behaviour (up
to polynomial time many-one reductions). This means, as far as their
complexity is concerned, it is not necessary to examine more CSPs than
there are finitely related clones on a given set.
\par

Hermann et
al.~\cite{CreignouHermannKrokhinSalzerComplexityOfClausalConstraintsOverChains}
studied CSPs involving so\dash{}called \emph{clausal constraints} over
totally ordered finite domains. Their problem can be understood within
the previously introduced CSP paradigm as a CSP given by a finite set of
relations of the form
\[\lset{(\liste{x}{p},\liste{y}{q})}{x_1\geq
a_1 \lor \dotsm \lor x_p\geq a_p \lor y_1\leq b_1\lor\dotsm\lor
y_q\leq b_q}.\]
These are called \emph{clausal relations} (over chains) and have been
studied more extensively
in~\cite{VargasCRelsCclones,%
         BehVargasCclonesAndCRelations,%
         BehVarUniqueInclusionsOfMaxCclones2018,%
         VargasMaxMinCMonoids,%
         Edith-thesis}.
\par
One of the main motivations for the CSPs studied
in~\cite{CreignouHermannKrokhinSalzerComplexityOfClausalConstraintsOverChains}
comes from many\dash{}valued logic, more precisely, from regular signed
logic over totally ordered finite sets of truth values, as
described in~\cite{HaehnleComplexityMVLogics,%
                   HaehnleAutomatedDeductionMVLogics}.
In fact, the satisfiability problem associated with regular signed
conjunctive normal form formulae over chains can be expressed as a CSP
over clausal relations (or with respect to clausal constraints).
In~\cite{CreignouHermannKrokhinSalzerComplexityOfClausalConstraintsOverChains}
a complete classification of complexity was achieved in terms of the
involved clausal patterns, establishing a dichotomy between tractability
and NP-completeness. The authors left open the problem to
algebraically describe all CSPs on the same domain whose complexity is
equivalent to one of their problems via Jeavons's reduction, with
particular focus on the tractable cases. This problem
is in fact asking for a description of all clones (with
tractable CSP) associated with clausal relations, called
\emph{clausal clones} in~\cite{VargasCRelsCclones}.
Since there is a continuum of clones on finite, at least
three\dash{}element sets, a first necessary step to ensure the
feasibility of answering such a question is to determine the cardinality
of the lattice of all clausal clones, which is another open problem
from~\cite{Edith-thesis},
see also~\cite[Section~1]{BehVarUniqueInclusionsOfMaxCclones2018}.%
\par

We address the open problem
from~\cite{CreignouHermannKrokhinSalzerComplexityOfClausalConstraintsOverChains}
by answering this feasibility question in the affirmative way. It turns
out that solving this problem is less convoluted when generalizing from
clausal relations to relations defined as disjunctions over arbitrary
subsets, not just upsets or downsets of finite total orders. In
this way, for finite carrier sets, we include in particular \emph{all} signed logics discussed
in~\cite{HaehnleComplexityMVLogics}: general regular signed logic with
respect to any order, monosigned logic and full signed logic.
Moreover, relations defined via disjunctions of unary relations have
seen applications in general algebra for more than three decades, see,
\eg, \cite{SzendreiIdempotentAlgebrasWithRestrictionsOnSubalgebras}. If~\m{n} is the arity of such a relation, it is also
known under the term \emph{\nbdd{n}dimensional
cross}~\cite{KearnesSzendreiCubeTermBlockersWithoutFiniteness,
             OprsalTaylorsModularityConjAndRelatedProblems}.
When the unary relations of a cross are subuniverses of a given algebra,
crosses have recently become prominent as a means to characterize the
non\dash{}existence of \emph{\nbdd{n}dimensional cube terms} in
idempotent
varieties~\cite{KearnesSzendreiCubeTermBlockersWithoutFiniteness}.
Symmetric crosses, \ie, disjunctions of only one non\dash{}trivial
subuniverse, have appeared earlier as so\dash{}called \emph{cube term
blockers}~\cite{MarkovicMarotiFinitelyRelatedClonesAndAlgebrasWithCubeTerms}.
\par
Our main result is that for any finite non\dash{}trivial set of unary
relations~\m{\Gamma} on any carrier set, there are exactly
\nbdd{\aleph_0}many polymorphism clones determined by relations that are
disjunctively definable over~\m{\Gamma} (\ie, by crosses over~\m{\Gamma}).
\par
We achieve this theorem by relating the number of such clones to the
number of downsets of a certain order on finite powers of the natural
numbers. The latter can be bounded above by~\m{\aleph_0} using
\name{Dickson}'s Lemma. Finally, we prove that our bound is tight by
exhibiting an infinite chain of clones whenever the parameter
set~\m{\Gamma} contains a non\dash{}empty relation~\m{\gamma} distinct
from the full domain.
\par
On a concluding note, we observe that our result extends to
non\dash{}trivial countably infinite sets of unary relations~\m{\Gamma},
when one is only interested in polymorphism clones of finite subsets of
disjunctively definable relations over~\m{\Gamma}.
\par
The sections of this paper should be read in consecutive order.
The following one introduces some notation and prepares basic
definitions and facts concerning clone and order theory.
Section~\ref{sect:clones-unary-disj} links polymorphism clones of relations
being disjunctively definable over unary relations with downsets of a poset
on~\m{\N^I}, and the final section concludes the task by counting these and
deriving our results.

\section{Preliminaries}\label{sect:prelim}
\subsection{Sets, functions and relations}\label{subsect:funrel}
We write \m{\N=\set{0,1,2,\dotsc}} for the set of \emph{natural numbers}
and use \m{\Np} for \m{\N\setminus\set{0}}. \emph{Inclusion} between
sets~\m{A} and~\m{B} is denoted by \m{A\subs B}, as opposed to \emph{proper
inclusion} \m{A\subset B} or \m{A\subsetneq B}. The \emph{powerset}
\m{\powerset{A}} is the set of all subsets of~\m{A}; \m{\powersetfin{A}}
the set of all finite subsets of~\m{A}. The \emph{cardinality} of~\m{A} is
written as~\m{\abs{A}}. It is convenient for us to use \name{John von
Neumann}'s model of~\m{\N} where each \m{n\in\N} is the set
\m{n=\set{0,\dotsc,n-1}}. If \m{\functionhead{f}{A}{B}} and
\m{\functionhead{g}{B}{C}} are functions, their \emph{composition} is the
map \m{\functionhead{g\circ f}{A}{C}} given by \m{g\circ f\apply{a} =
g\apply{f\apply{a}}} for all \m{a\in A}. Moreover, if \m{U\subs A} and
\m{V\subs B}, we write \m{f\fapply{U}} for the \emph{image}
\m{\lset{f\apply{u}}{u\in U}} of~\m{U} under~\m{f} and
\m{f^{-1}\fapply{V}} for the \emph{preimage}
\m{\lset{a\in A}{f\apply{a}\in V}} of~\m{V} with respect to~\m{f}. The
\emph{image of\/~\m{f}} is denoted by
\m{\im{f}\defeq f\fapply{\CarrierSet}}; the \emph{kernel of\/~\m{f}},
denoted by \m{\ker f}, is the equivalence relation
\m{\lset{\apply{a_1,a_2}\in\CarrierSet[2]}{f\apply{a_1}=f\apply{a_2}}}.
Clearly, every member~\m{f\apply{a}} of the image of~\m{f} is in
one\dash{}to\dash{}one correspondence with the kernel class
\m{\fapply{a}_{\ker f}}, \ie, we have a bijection between \m{\im{f}} and
the \emph{factor set} \m{\factorBy{A}{\ker f}} of all equivalence
classes by the kernel of~\m{f}.
\par
If~\m{I} and~\m{\CarrierSet} are sets, the direct (Cartesian) power
\m{A^I} is the set of all maps \m{\functionhead{f}{I}{\CarrierSet}}. This
is also true for finite powers, \ie\ \m{\CarrierSet[n]} where \m{n\in \N}:
we understand \nbdd{n}tuples over~\m{\CarrierSet} as maps from~\m{n}
to~\m{\CarrierSet}; in particular all notions defined for maps, such as
composition, image, preimage, kernel also make sense for tuples. Formally,
an \nbdd{n}tuple \m{x\in\CarrierSet[n]} is given as
\m{x=\apply{x\apply{0},\dotsc,x\apply{n-1}}}, however, if no confusion is
to be expected, we shall also refer to the entries of a tuple by some
other indexing, \eg\ \m{x= \apply{\liste{x}{n}}} or \m{x= \apply{a,b,c}}
etc. Moreover, an \emph{\nbdd{n}ary relation on~\m{\CarrierSet}} is just
any subset \m{\rho\subs\CarrierSet[n]} of \nbdd{n}tuples; an \nbdd{n}ary
operation on~\m{\CarrierSet} is any
function~\m{\functionhead{f}{\CarrierSet[n]}{\CarrierSet}}.
We collect all \emph{finitary (excluding nullary) operations
on~\m{\CarrierSet}} in the set
\m{\Ops=\bigcup_{n\in\Np}\CarrierSet[{A^n}]}.
Multi\dash{}ary functions \m{\functionhead{f}{B^n}{C}} and
\m{\functionhead{\liste{g}{n}}{\CarrierSet[m]}{B}} can be
composed in the following way: putting
\m{\composition{f}{\liste{g}{n}}\apply{a} \defeq
    f\apply{g_1\apply{a},\dotsc,g_n\apply{a}}}
for every \m{a\in\CarrierSet[m]} defines a function
\m{\functionhead{\composition{f}{\liste{g}{n}}}{\CarrierSet[m]}{C}}.
This works for functions and tuples, too.
Namely, we say that
an \nbdd{n}ary function \m{f} \emph{preserves} an \nbdd{m}ary relation
\m{\rho\subs\CarrierSet[m]} and write \m{f\preserves\rho} if for all \m{\apply{\liste{r}{n}}\in\rho^n}
we automatically have \m{\composition{f}{\liste{r}{n}}\in\rho}.
Then for a set \m{Q\subs\bigcup_{m\in\Np}\powerset{\CarrierSet[m]}} of
finitary relations on~\m{\CarrierSet}, we define
\m{F=\Pol{Q}\defeq\lset{f\in\Ops}{\forall \rho\in Q\colon  f\preserves
\rho}} and call this the set of \emph{polymorphisms} of~\m{Q}. This set
of operations forms a \emph{clone on~\m{\CarrierSet}}, \ie, it is closed
under composition (viz.\ \m{\composition{f}{\liste{g}{n}}\in F} whenever
\m{f\in F} is \nbdd{n}ary and \m{\liste{g}{n}\in F} are \nbdd{m}ary) and
contains all projections
\m{\functionhead{\eni{i}}{\CarrierSet[n]}{\CarrierSet}}, \m{0\leq i<
n}, \m{n\in\Np}, given by \m{\eni{i}\apply{\listen{x}{n}} = x_i} for
\m{\apply{\listen{x}{n}}\in\CarrierSet[n]}.
A clone on~\m{\CarrierSet} is \emph{finitely related} if it can be
obtained as \m{\Pol{Q}} for a finite set~\m{Q} of finitary relations.
More information on the importance of finitely related clones in the
context of CSP can be found
in~\cite{BartoKrokhinWillard_PolymorphismsHowTo}.
\par
Note that the preservation relation~\m{\preserves} between finitary
operations and relations, and the Galois correspondence induced by it,
is fundamental for the study of clones on finite
sets~\cite{PoeKal,SzendreiClonesInUniversalAlgebra}.
In point of fact, this paper is concerned with counting the number of
Galois closed sets for a variant of this Galois correspondence where the
relational side is restricted to the subset~$\DD$ of crosses over
a given set~$\Gamma$ of unary relations. For more background information
and basic facts concerning Galois connections in general,
see~\cite[p.~155 et seqq.]{DaveyPriestley} or~\cite{GaWiFCA}.

\subsection{Ordered sets}\label{subsect:posets}
If \m{\poset} is a \emph{partially ordered set}
(\emph{poset}), \ie\ \m{\ordrel} is a binary reflexive, antisymmetric and
transitive relation on \m{\genericOrder}, then a subset \m{X\subs
\genericOrder} is said to be a \emph{downset of\/~\m{\order}} (occasionally
called \emph{order ideal}), if it is closed \wrt\ taking lower bounds. This
means that with every member \m{x\in X} the \emph{principal
downset \m{\downset{\set{x}}\defeq\lset{y\in \genericOrder}{y\lt x}}
generated by~\m{x}} is a subset of~\m{X}. We denote the \emph{set of all
downsets of\/~\m{\order}} by \m{\DS}.
\par

It is easy to see that \m{\DS} forms a closure system on~\m{\genericOrder},
the associated closure operator maps any set \m{Y\subs\genericOrder} to its
closure under lower bounds
\m{\downset{Y}\defeq\bigcup_{y\in Y} \downset{\set{y}}
             = \lset{z\in \genericOrder}{\exists y\in Y\colon z \lt y}},
\ie\ the least downset of \m{\order} containing \m{Y}.
This set is also referred to as \emph{downset generated by~\m{Y}}.
Clearly, a set \m{Y\subs\genericOrder} is a downset if and only if
\m{\downset{Y} = Y}.
\par

The dual notion of a downset is that of an \emph{upset} (\emph{order
filter}), which is a subset \m{X\subs\genericOrder} of a poset \m{\poset}
that is closed under upper bounds. Again the collection \m{\US} of all
upsets of \m{\order} forms a closure system and the corresponding closure
operator \m{\upset{}} is given by adding all upper bounds. Obviously,
complementation establishes a one\dash{}to\dash{}one correspondence between
\m{\DS} and \m{\US}.
\par

A \emph{(homo)morphism} from a poset~\m{\order} to~\m{\order[Q]} is any
monotone map \m{\functionhead{h}{\genericOrder}{Q}}, \ie\ one being
compatible with the respective order relations.
A morphism \m{\functionhead{h}{\order}{\order[Q]}} is a \emph{retraction} if there
exists a homomorphism \m{\functionhead{\tilde{h}}{\order[Q]}{\order}}
that is a right\dash{}inverse to~\m{h}, \ie, satisfies \m{h\circ\tilde{h}= \id_{\order[Q]}}.
An \emph{isomorphism} between posets~\m{\order} and~\m{\order[Q]} is a
retraction \m{\functionhead{h}{\order}{\order[Q]}} that also has a
left\dash{}inverse, \ie\ any order preserving and order reflecting
bijection.
\par

We note the following basic lemmas.
\begin{lemma}\label{lem:retr-maps-ds-to-ds}
If \m{\functionhead{h}{\order}{\order[Q]}} is a homomorphism between
posets\/~\m{\order} and\/~\m{\order[Q]} and a subset
\m{Y\in\DS[{\order[Q]}]} is a downset,
then so is the preimage \m{h^{-1}\fapply{Y}
\in\DS[{\order}]}.
\end{lemma}
\begin{proof}
Let \m{x\in h^{-1}\fapply{Y}} and \m{z\in\order} such that \m{z\lt x}.
As~\m{h} is monotone, we have \m{h\apply{z}\leq h\apply{x} \in Y}. As
\m{Y\in\DS[{\order[Q]}]}, we get \m{h\apply{z}\in Y}, and thus
\m{z\in h^{-1}\fapply{Y}}.
\end{proof}

\begin{lemma}\label{lem:iso-bij-DS}
If \m{\functionhead{h}{\order}{\order[Q]}} is an isomorphism between
posets\/~\m{\order} and\/~\m{\order[Q]}, then
\m{\functionhead{H}{\DS}{\DS[{\order[Q]}]}} given by \m{H\apply{X}\defeq
h\fapply{X}} is a bijection. In particular we have\/
\m{\abs{\DS}=\abs{\DS[{\order[Q]}]}}.
\end{lemma}
\begin{proof}
By Lemma~\ref{lem:retr-maps-ds-to-ds} the map~\m{H} is well\dash{}defined
because \m{H\apply{X}} is the preimage of \m{X} under the inverse
isomorphism \m{\functionhead{h^{-1}}{\order[Q]}{\order}} of~\m{h}.
The inverse~\m{H^{-1}} is given by \m{H^{-1}\apply{Y} = h^{-1}\fapply{Y}} for
\m{Y\in\DS[{\order[Q]}]}.
\end{proof}
\par

If \m{\poset} is a poset and \m{Y\subs\genericOrder} is a subset, then
\m{\mathord{\restrordrel{Y}}
                           \defeq \mathord{\lt}\cap\apply{Y\times Y}}
denotes the \emph{restricted order relation}, \ie\ the order of the
\emph{subposet\/ \m{\ordset{Y}{Y}{\restrordrel{Y}}} of\/~\m{\order}
induced by \m{Y}}.
\par

\begin{lemma}\label{lem:ds-to-subposets}
If\/~\m{\order} is a poset and \m{X,Y\in\DS}, then
\m{X\cap Y\in\DS[{\order[Y]}]} where
\m{\ordset{Y}{Y}{\restrordrel{Y}}}.
\end{lemma}
\begin{proof}
Downsets are closed with regard to intersections, so \m{X\cap Y} clearly
belongs to \m{\DS}. As \m{\order[Y]} carries the restricted order
of~\m{\order}, we get \m{X\cap Y\in\DS[{\order[Y]}]}.
\end{proof}
\par

Next, we consider a special order on powers of \m{\N}. For any set~\m{I} we
denote the order of \m{\apply{\N, \leq}^I}, which is given by requiring the
relationship~\m{\leq} pointwise, also by~\m{\leq}. For any element
\m{x\in\N^I}, we refer by its \emph{support} to the set
\m{\supp{x}\defeq \lset{i\in I}{x\apply{i}\neq 0}}, \ie\ the preimage
\m{x^{-1}\fapply{\N\setminus\set{0}}}.
Relating two elements \m{x,y\in\N^I} if their supports are dually
contained defines a quasiorder on \m{\N^I}, \ie\ a reflexive and transitive
binary relation. Intersecting this quasiorder with the pointwise
order~\m{\leq}, we obtain the poset
\m{\ordset{\genericOrder}{\N^I}{\below}}, where \m{x\below y} holds for
\m{x,y\in\N^I} if and only if \m{x\leq y} and \m{\supp{y}\subs\supp{x}}. If
\m{x\below y} then for every \m{i\in\supp{x}} we have
\m{0<x\apply{i}\leq y\apply{i}}, \ie\ \m{i\in\supp{y}}. Hence, the
condition \m{x\below y} is equivalent to \m{x\leq y} and
\m{\supp{x}=\supp{y}}.
\par

In the following sections, we shall be interested in the downsets of the
poset \m{\apply{\N^I,\below}}, mostly for finite~\m{I}. In order to count
these we need information about the number of downsets of finite powers of
\m{\apply{\N,\leq}}, which is related to \emph{well\dash{}partial orders}.
\begin{definition}[{\eg~\cite[Definition~5.4.3, p.~113]{%
                   BaaderNipkowTermRewritingAndAllThat}}]%
                   \label{def:wpo}
A poset \m{\poset} is called a \emph{well\dash{}partial order} if for every
sequence \m{\apply{x_i}_{i\in \N}\in\genericOrder^{\N}} there exist indices
\m{i<j} such that \m{x_i\leq x_j}.
\end{definition}

With regard to such orders we note the following basic lemma, which can
also be found in the monograph~\cite{BaaderNipkowTermRewritingAndAllThat}.
\begin{lemma}[{\cite[Lemma~5.4.5, p.~113]{%
               BaaderNipkowTermRewritingAndAllThat}}]%
              \label{lem:dir-prod-of-wpos}
Any finite direct product of well\dash{}partial orders is again a
well\dash{}partial order.
\end{lemma}

A poset \m{\poset} satisfies the \emph{descending chain condition}
(\emph{DCC}) if every countably infinite descending sequence
\m{\apply{x_i}_{i\in\N}\in P^\N} eventually stabilizes, \ie\ whenever
\m{x_j\lt x_i} holds for all \m{i\leq j}, then there exists some \m{i\in\N}
such that \m{x_j=x_i} for all \m{j\geq i}.
\par

We readily observe the following consequence concerning descending chains and
anti\dash{}chains, \ie\ subsets of a poset consisting entirely of pairwise
incomparable elements.
\begin{lemma}\label{lem:wpo-no-inf-anti-chain}
If\/~\m{\poset} is a well\dash{}partial order, then\/~\m{\order} does
not contain infinite anti\dash{}chains and satisfies DCC.
\end{lemma}
\begin{proof}
If \m{A\subs\genericOrder} is an infinite subset such that
\m{\apply{A,\restrordrel{A}}} is an anti\dash{}chain, then there exists a
sequence \m{a = \apply{a_i}_{i\in\N}\in A^\N} such that
\m{\lset{a_i}{i\in\N}\subs A} is a countably infinite subset of \m{A}, and
hence an anti\dash{}chain, too. Therefore, \m{a} violates the defining
property of a well\dash{}partial order.
\par
If\/~\m{\order} did not satisfy DCC, there would be an infinite
descending sequence \m{y\in \genericOrder^{\N}} that does not eventually
stabilise. Then one could construct a
strictly decreasing subsequence \m{x = \apply{x_i}_{i\in\N}} of~\m{y}
satisfying \m{x_j<x_i} for all \m{i\leq j}, in contradiction
to~\m{\order} being a well\dash{}partial order.
\end{proof}
\par

The properties occurring in Lemma~\ref{lem:wpo-no-inf-anti-chain}
suffice to bound the number of downsets of a countable ordered
set~\m{\order}.
\begin{lemma}\label{lem:DCC-finite-antichains-cntbl-implies-cntbl-ds}
If\/~\m{\poset} is a poset on a countable (\ie finite or countably
infinite) set
\m{\genericOrder} satisfying DCC and not containing any infinite
anti\dash{}chains, then the number of downsets of\/~\m{\order} is
countable.
\end{lemma}
\begin{proof}
Complementation bijectively maps downsets of \m{\order} to upsets and vice
versa. Hence, consider any \m{F\in\US}; its set of minimal elements
\m{M\apply{F}} certainly forms an anti\dash{}chain, so by the assumption,
\m{M\apply{F}\subs \genericOrder} is a finite subset. Moreover, it is an
easy consequence of DCC that every element \m{x\in F} satisfies \m{m\leq x}
for some \m{m\in M\apply{F}}. In other words, we have
\m{F = \upset{M\apply{F}}}, which shows that the map
\m{\functionhead{M}{\US}{\powersetfin{\genericOrder}}} is injective.
Finally, the set \powersetfin{\genericOrder} of all finite subsets of
\m{\genericOrder} is countable since for every \m{k\in\N}
the set of all \nbdd{k}element subsets of \m{\genericOrder} is embeddable
into \m{\genericOrder^k}, whose cardinality is bounded by \m{\aleph_{0}}
using the bound on the cardinality of~\m{\genericOrder}.
\end{proof}

\begin{corollary}\label{cor:ds-of-dcc-cntbl-wpo}
Any well\dash{}partial order\/~\m{\poset} on a countable
carrier set~\m{\genericOrder} has countably many downsets.
\end{corollary}
\begin{proof}
Applying Lemma~\ref{lem:wpo-no-inf-anti-chain} to the assumptions yields
the premiss of
Lemma~\ref{lem:DCC-finite-antichains-cntbl-implies-cntbl-ds}, which entails
the claim.
\end{proof}
\par

Combining the previous observations for the case of finite powers of
\m{\apply{\N,\leq}} gives the following result.
\begin{lemma}\label{lem:inf-N-to-I}
For every finite set~\m{I} the poset \m{\apply{\N,\leq}^I} is a
well\dash{}partial order, does not have any infinite anti\dash{}chains
(\name{Dickson}'s Lemma) and has a countable number of downsets.
\end{lemma}
\begin{proof}
First, observe that \m{\apply{\N,\leq}} is a well\dash{}partial order:
namely, for every sequence \m{\apply{n_i}_{i\in \N}\in\N^\N} of pairwise
distinct elements the downset \m{\downset[{\apply{\N,\leq}}]{\set{n_0}}} is
finite, so there is a largest index \m{i\in\N} such that
\m{n_i\in\downset[{\apply{\N,\leq}}]{\set{n_0}}}. However, this means
\m{n_j\geq n_0}, even \m{n_j>n_0}, for all \m{j >i}.
\par
Based upon this, finiteness of \m{I} ensures that \m{\apply{\N,\leq}^I} is
a well\dash{}partial order by Lemma~\ref{lem:dir-prod-of-wpos}; hence
\name{Dickson}'s Lemma follows from Lemma~\ref{lem:wpo-no-inf-anti-chain}.
Moreover, \m{\abs{I}<\aleph_{0}} entails that
\m{\abs{\N^I}\leq \aleph_{0}}, whence
Lemma~\ref{lem:DCC-finite-antichains-cntbl-implies-cntbl-ds} or
Corollary~\ref{cor:ds-of-dcc-cntbl-wpo} bounds the number of downsets
from above.
\end{proof}
\par

\begin{corollary}\label{cor:downsets-of-N-to-I}
We have\/ \m{\abs{\DS[{\apply{\N,\leq}^I}]}=\aleph_{0}}
for all finite \m{I\neq \emptyset}.
\end{corollary}
\begin{proof}
The inequality \m{\abs{\DS[{\apply{\N,\leq}^I}]}\leq \aleph_{0}} follows
from Lemma~\ref{lem:inf-N-to-I}; for the converse it suffices to exhibit an
infinite subset of \m{\DS[{\apply{\N,\leq}^I}]}, for instance, infinitely
many principal downsets
\m{\lset{\downset{\set{\apply{n,\dotsc,n}}}}{n\in \N}}.
\end{proof}
\par

\section{Clones determined by disjunctions of unary predicates}%
\label{sect:clones-unary-disj}
We are interested in clones that are determined by relations that are
disjunctively definable by unary predicates. Throughout our whole study
we shall fix a (mostly finite) parameter set~\m{\Gamma} of unary relations on a
given carrier set~\m{\CarrierSet} that we call \emph{(unary) relational
language}. If \m{n\in\Np} and \m{\apply{\liste{\gamma}{n}}\in \Gamma^n}
is any \nbdd{n}tuple of basic relations from~\m{\Gamma},
then every \nbdd{n}ary relation \m{\rho} of the form
\m{\rho = \Runary{\liste{\gamma}{n}}
=\lset{\apply{\liste{x}{n}}\in\CarrierSet[n]}{%
      \bigvee_{1\leq i\leq n} x_i\in\gamma_i}}
is said to be \emph{disjunctively
definable from \m{\Gamma}}. Relations constructed in this way have been
denoted as \m{\Cr\apply{\liste{\gamma}{n}}}
in~\cite{KearnesSzendreiCubeTermBlockersWithoutFiniteness,
         OprsalTaylorsModularityConjAndRelatedProblems}.
The set \m{\DD} consists exactly of all
relations definable in this manner. Note that \m{\DD[\emptyset]=\emptyset}
since we consider only non\dash{}empty disjunctions, \ie, disjunctively
definable relations of positive arity.
\par
We quickly observe that \m{\Runary{\liste{\gamma}{n}}=\CarrierSet[n]}
holds if and only if at least one of the relations \m{\liste{\gamma}{n}}
equals the full set~\m{\CarrierSet}. Namely, if for every
\m{i\in\set{1,\dotsc,n}} the basic unary relation \m{\gamma_i\subsetneq
\CarrierSet} is proper, and \m{x_i\in \CarrierSet\setminus\gamma_i}, then
\m{\apply{\liste{x}{n}}\notin \Runary{\liste{\gamma}{n}}}, so
\m{\Runary{\liste{\gamma}{n}}\subsetneq\CarrierSet[n]}.
Likewise, we note that \m{\Runary{\liste{\gamma}{n}}=\emptyset} if and
only if \m{\gamma_1= \dotsm = \gamma_n=\emptyset}.
\par
It is useful to see that the basic unary relations defining some
\nbdd{n}ary relation \m{\rho\in\DD} can be uniquely reconstructed
from~\m{\rho} in every interesting case, \ie\ whenever
\m{\rho\neq \CarrierSet[n]}.

\begin{lemma}\label{lem:parameter-reconstruction}
The parameter reconstruction map
\[\function{p}{\DD}{\bigcup_{n\in\Np}\Gamma^n}{%
\rho=\Runary{\liste{\gamma}{n}}}{%
\begin{cases}
\apply{\liste{\gamma}{n}} & \text{ if } \rho\subsetneq\CarrierSet[n]\\
\apply{A,\dotsc,A} & \text{ else,}
\end{cases}}\]
is well\dash{}defined.
\end{lemma}
\begin{proof}
By definition of~\m{\DD} every disjunctively definable relation \m{\rho}
has a parameter representation \m{\rho = \Runary{\liste{\gamma}{n}}}. We
need to prove that the latter is unique, whenever
\m{\rho\neq\CarrierSet[n]}. For this consider
\m{\apply{\liste{\gamma}{n}}, \apply{\liste{\gamma'}{n}}\in\Gamma^n}
such that
\m{\Runary{\liste{\gamma}{n}} = \Runary{\liste{\gamma'}{n}}
\subsetneq\CarrierSet[n]}. We now show that \m{\gamma_i=\gamma'_i}
holds for all \m{i\in\set{1,\dotsc,n}}. Since our assumption is
symmetric, it suffices to fix \m{i\in\set{1,\dotsc,n}} and to prove that
\m{\gamma_i\subs\gamma'_i}.
Because \m{\rho=\Runary{\liste{\gamma'}{n}}\neq \CarrierSet[n]}, we have
\m{\gamma'_j\subsetneq A} for all \m{j\in\set{1,\dotsc,n}} and so we can
pick \m{x_j\in A\setminus\gamma'_j} for \m{j\neq i}. If now
\m{x_i\in\gamma_i}, we have
\m{\apply{\liste{x}{n}}\in\Runary{\liste{\gamma}{n}} =
\Runary{\liste{\gamma'}{n}}}. By the choice of the \m{x_j} for all
\m{j\neq i}, we must have \m{x_i\in \gamma'_i}, proving
\m{\gamma_i\subs\gamma'_i}.
\end{proof}

It is an obvious consequence of the preceding lemma that the parameter
reconstruction map provides a one-sided inverse to the construction of
relations from unary predicates. That is to say, we have
\m{\Runary{p\apply{\rho}} = \rho} for all \m{\rho\in\DD}. This inverse is
uniquely determined for the non\dash{}trivial relations~\m{\rho}, and it
chooses a canonical representative of all possible parametrizations
when~\m{\rho} is a full power of~\m{\CarrierSet}.
\par

Based on the parameter reconstruction~\m{p} from
Lemma~\ref{lem:parameter-reconstruction}, we can define the
\emph{pattern} of a disjunctively definable relation. Intuitively it
counts how often each \m{\gamma\in\Gamma} occurs in the parameter tuple
\m{p\apply{\rho}}.
More formally, for \m{\rho\in\DD} with \m{p\apply{\rho} =
\apply{\liste{\gamma}{n}}}, \ie\ \m{\rho=\Runary{\liste{\gamma}{n}}},
the tuple \m{\pt{\rho}\in\N^{\Gamma}} maps every \m{\gamma\in\Gamma} to
\m{\pt{\rho}\apply{\gamma} \defeq
                    \abs{p\apply{\rho}^{-1}\fapply{\set{\gamma}}}}
where for both the parameter tuple \m{p\apply{\rho}\in \Gamma^n} and the
pattern \m{\pt{\rho}\in \N^\Gamma} we make use of the ambiguous
interpretation as a tuple and as a map. Note that the length of the
parameter tuple \m{p\apply{\rho}} depends on the arity of \m{\rho},
while the length of the pattern only depends on~\m{\abs{\Gamma}}, which
is usually finite (at least \m{\pt{\rho}} has finite support
in~\m{\Gamma}) and normally not varying for our considerations.
The pattern of a disjunctively definable relation roughly carries the
same information as \emph{clausal
patterns}~\cite[Section~2]{CreignouHermannKrokhinSalzerComplexityOfClausalConstraintsOverChains}
do for clausal relations; however, since we are dealing with a more
generic situation, the notion of clausal pattern had to be
adapted and generalized.
\par

Next we study how the polymorphism clone of a disjunctively definable
relation~\m{\rho} over~\m{\Gamma} changes when duplicating one of its unary
parameter relations. Since polymorphism clones are not affected by variable
permutations of their defining relations, we can restrict our attention
to duplicating the first parameter of \m{\rho\in \DD}.
Letting \m{p\apply{\rho} = \apply{\liste{\gamma}{n}}}, such
a duplication apparently increases the pattern of
\m{\rho =\Runary{\liste{\gamma}{n}}} in precisely one place, while
preserving all the other values:
\[
\pt{\Runary{\gamma_1,\liste{\gamma}{n}}}\apply{\gamma} =
\begin{cases}
\pt{\Runary{\liste{\gamma}{n}}}\apply{\gamma}+1&
                                          \text{if }\gamma=\gamma_1,\\
\pt{\Runary{\liste{\gamma}{n}}}\apply{\gamma} & \text{otherwise.}
\end{cases}\]

\begin{lemma}\label{lem:duplicating-parameters}
For arbitrary unary relations
\m{\liste{\gamma}{n}\subs\CarrierSet} we have
\[\Pol{\set{\Runary{\gamma_1,\liste{\gamma}{n}}}}
 \subs\Pol{\set{\Runary{\liste{\gamma}{n}}}}.\]
\end{lemma}
\begin{proof}
Let \m{k\in\N} and
\m{f\in\Pol{\set{\Runary{\gamma_1,\liste{\gamma}{n}}}}} be a \nbdd{k}ary
polymorphism and \m{r_1,\dotsc,r_k\in\Runary{\liste{\gamma}{n}}}.
Let \m{s_j\in \CarrierSet[n+1]} arise from \m{r_j} by duplicating the
first entry of the tuple (for each \m{j\in\set{1,\dotsc,k}}). Then,
clearly, we have \m{\liste{s}{k}\in\Runary{\gamma_1,\liste{\gamma}{n}}},
thus
\m{x\defeq \composition{f}{\liste{s}{k}}\in\Runary{\gamma_1,\liste{\gamma}{n}}}
as \m{f} preserves \m{\Runary{\gamma_1,\liste{\gamma}{n}}}.
Since the first and second entry of~\m{x} are identical and the last~\m{n}
entries of~\m{x} coincide with \m{y\defeq \composition{f}{\liste{r}{k}}},
we obtain \m{y\in\Runary{\liste{\gamma}{n}}}.
Therefore, we have \m{f\in\Pol{\set{\Runary{\liste{\gamma}{n}}}}}.
\end{proof}

Since the preservation property remains unaffected by variable
permutations of relations, we have the following immediate corollary.

\begin{corollary}\label{cor:duplicating-permuting-parameters}
If \m{\rho,\rho'\in\DD} are such that~\m{\rho'} arises from~\m{\rho}
by a finite number of applications of the operations of
duplicating some parameters and
of rearranging the order of parameters,
then \m{\Pol{\set{\rho'}}\subs\Pol{\set{\rho}}}.
\end{corollary}

Observe that if \m{\rho'} arises from \m{\rho\in\DD} as described in the
previous corollary, then \m{\pt{\rho}\leq\pt{\rho'}} holds with
respect to the
pointwise order of tuples. The increases occur exactly in the places
where parameter relations have been duplicated. The only exception to this
is the case when \m{\rho} is a full power of~\m{\CarrierSet}, where we may
duplicate some (non\dash{}canonical) parameter relation
\m{\gamma_1\neq \CarrierSet}, but observe an increase of
\m{\pt{\rho}\apply{\CarrierSet}}. This issue does not occur if we only
duplicate canonical parameters as computed by \m{p\apply{\rho}}.
\par
Moreover, in this
process, no new basic unary relations can be introduced for~\m{\rho'}
that have not already been present as parameters of \m{\rho}. This means,
if \m{\pt{\rho}\apply{\gamma} = 0} for some \m{\gamma\in\Gamma}, the same
must be true for \m{\pt{\rho'}\apply{\gamma}}.
In other words, \m{\supp{\pt{\rho'}}\subs \supp{\pt{\rho}}}.
Combining these observations we conclude that \m{\pt{\rho}
\below\pt{\rho'}} must be satisfied when transmuting
\m{\rho\leadsto\rho'}.
\par
In fact the converse is also true, which is the reason for the following
crucial lemma, relating the order~\m{\below} on patterns of disjunctively
definable relations and the inclusion of their corresponding polymorphism
clones.
Note that, as Corollary~\ref{cor:duplicating-permuting-parameters}, this
lemma does not really depend on the finiteness of~\m{\Gamma}, it
only depends on the
fact that \m{\supp{\pt{\rho}}, \supp{\pt{\rho'}}\subs\Gamma} are finite.
\begin{lemma}\label{lem:downset-implies-Pol}
For \m{\rho,\rho'\in\DD} satisfying \m{\pt{\rho}\below\pt{\rho'}} we
have the dual inclusion \m{\Pol{\set{\rho'}}\subs\Pol{\set{\rho}}}.
\end{lemma}
\begin{proof}
If \m{\rho,\rho'\in\DD} are such that \m{\pt{\rho}\below\pt{\rho'}},
then by verifying that the assumptions of
Corollary~\ref{cor:duplicating-permuting-parameters} are fulfilled, we
see that \m{\Pol{\set{\rho'}}\subs\Pol{\set{\rho}}}.
In more detail, from \m{\pt{\rho}\below\pt{\rho'}} we get that
\m{\supp{\pt{\rho}}=\supp{\pt{\rho'}}}, so in the parameter tuples
\m{p\apply{\rho}} and \m{p\apply{\rho'}} the same relations
from~\m{\Gamma} occur. Because \m{\pt{\rho}\leq\pt{\rho'}}, those that
are actually present, occur possibly a few more times in the pattern of
\m{\rho'} than in that of~\m{\rho} and are perhaps associated with
different coordinates in~\m{\rho} and~\m{\rho'}. However this means
exactly that \m{\rho'} can be obtained from \m{\rho} by duplicating and
permuting parameters in a finite number of steps (since
\m{\supp{\pt{\rho}}= \supp{\pt{\rho'}}} is finite).
\end{proof}

Now, finally, in order to bound the number of polymorphism clones given by
sets~\m{Q} of relations that are disjunctively definable from~\m{\Gamma},
we associate with each such set a downset of
\m{\apply{\N^\Gamma,\below}}, namely the downset generated by the
associated patterns. More formally, we define the encoding
\[\function{I}{\powerset{\DD}}{\DS[{\N^{\Gamma},\below}]}{%
   Q}{\downset[{\apply{\N^\Gamma,\below}}]{\ptOp\fapply{Q}}
  =\bigcup_{\rho\in Q}\downset[{\apply{\N^\Gamma,\below}}]{\pt{\rho}}.}\]

With the help of Lemma~\ref{lem:downset-implies-Pol} we can now prove the
following relationship:
\begin{proposition}\label{prop:encoding-implies-Pol}
For \m{Q_1,Q_2\subs\DD} satisfying \m{I\apply{Q_1}\subs I\apply{Q_2}} we
have the inclusion \m{\Pol{Q_2}\subs\Pol{Q_1}}.
\end{proposition}
\begin{proof}
For \m{Q_1,Q_2\subs\DD} assume \m{I\apply{Q_1}\subs I\apply{Q_2}}. Let
\m{f\in\Pol{Q_2}} and a relation \m{\rho\in Q_1} be chosen arbitrarily.
Abbreviating \m{\order \defeq \apply{\N^\Gamma,\below}}, we have
\[\pt{\rho}\in
\downset{\ptOp\fapply{Q_1}} = I\apply{Q_1} \subs I\apply{Q_2}
=\downset{\ptOp\fapply{Q_2}}
=\bigcup_{\rho'\in Q_2}\downset{\pt{\rho'}}.
\]
Hence, there is some \m{\rho'\in Q_2} such that
\m{\pt{\rho}\in\downset{\pt{\rho'}}}, \ie, \m{\pt{\rho}\below\pt{\rho'}}.
Now Lemma~\ref{lem:downset-implies-Pol} implies
\m{f\in \Pol{Q_2}\subs\Pol{\set{\rho'}}\subs\Pol{\set{\rho}}}.
\end{proof}

More important for our target is actually the image
\m{\lset{\Pol{Q}}{Q\subs\DD}} of the map
\m{\functionhead{\PolOp}{\powerset{\DD}}{\lset{\Pol{Q}}{Q\subs\DD}}},
whose cardinality is linked to that of the factor set
\m{\factorBy{\powerset{\DD}}{\ker\PolOp}} by the kernel. As a
straightforward corollary of Proposition~\ref{prop:encoding-implies-Pol}
the latter is closely related to the kernel of~\m{I}.

\begin{corollary}\label{cor:kerI-subs-kerPol}
On\/ \m{\powerset{\DD}} we have the following inclusion between
equivalences:
\m{\ker I =
\lset{\apply{Q_1,Q_2}\in \powerset{\DD}^2}{I\apply{Q_1}=I\apply{Q_2}}
\subs\ker \PolOp}.
\end{corollary}

This result allows us to establish an upper bound on the number of clones
determined by relations that are disjunctively definable over~\m{\Gamma}.
\begin{corollary}\label{cor:upper-bound}
We have\/
\m{\abs{\im{\PolOp}}\leq \abs{\im{I}}\leq\abs{\DS[{\N^\Gamma,\below}]}}.
\end{corollary}
\begin{proof}
Since \m{\ker{I}\subs \ker\PolOp}, there is a canonical
well\dash{}defined surjection from \m{\factorBy{\powerset{\DD}}{\ker{I}}}
onto \m{\factorBy{\powerset{\DD}}{\ker{\PolOp}}}, so
\[
\im\PolOp \cong
\factorBy{\powerset{\DD}}{\ker{\PolOp}}
\twoheadleftarrow \factorBy{\powerset{\DD}}{\ker{I}} \cong \im I
\subs \DS[{\N^{\Gamma},\below}],\]
telling us that the cardinality of
\m{\abs{\im\PolOp} =\abs{\factorBy{\powerset{\DD}}{\ker{\PolOp}}}}
is bounded above by that of
\m{\abs{\factorBy{\powerset{\DD}}{\ker{I}}} = \abs{\im I}
   \leq \abs{\DS[{\N^{\Gamma},\below}]}}.
\end{proof}

An alternative proof of the previous fact can be obtained by noting that
the following map~\m{\psi}, representing disjunctively definable clones
as downsets of patterns, is injective. Moreover, it even embeds the whole
ordered structure of such clones.
\begin{proposition}\label{prop:upper-bound-by-injection}
The map
\[\function{\psi}{\apply{\lset{\Pol{Q}}{Q\subs\DD},\supseteq}}{\apply{\DS[{\N^\Gamma,\below}],\subs}}{F=\Pol{Q}}{\ptOp\fapply{F'},}
\]
where \m{F^{\inV}=\lset{\rho\in\DD}{\forall f\in F\colon f\preserves\rho}},
is a well\dash{}defined order embedding, and it makes the following
diagram commute:
\begin{center}%
\begin{tikzpicture}[set/.style={rectangle,fill=none,draw=none},x=9em]
\node[set] (pdd) at (0,1) {$\powerset{\DD}$};
\node[set] (ds)  at (1,1) {$\DS[{\N^\Gamma,\below}]$};
\node[set] (ip1) at (0,0) {$\im{\PolOp}$};
\node[set] (ip2) at (1,0) {$\im{\PolOp}$};
\draw[->,>=stealth] (pdd) edge node[above]{$I$} (ds)
                          edge node[auto]{$\PolOp$} (ip1)
                    (ds)  edge node[above]{$\phi$} (ip1)
                          edge[<-] node[auto]{$\psi$} (ip2)
                    (ip2) edge node[below] {$\id_{\im{\PolOp}}$} (ip1);
\end{tikzpicture}%
\end{center}
where~$\phi$ is any factor map (cf.~Corollary~\ref{cor:kerI-subs-kerPol}) satisfying
$\phi\apply{I\apply{Q}}= \Pol{Q}$ on the image of~$I$ and being defined
as, \eg, $\phi\apply{U}=\Pol{\emptyset}$ anywhere else.
\end{proposition}
\begin{proof}
In this proof, let us abbreviate
\m{Q^{\poL}\defeq\Pol{Q}} for any set~\m{Q\subs\DD}.
In order to demonstrate that~\m{\psi} is well\dash{}defined, we need to
show that \m{\ptOp\fapply{F^{\inV}}\subs\N^\Gamma} is a downset with
respect to~\m{\mathord{\below}} given \m{F=Q^{\poL}} for some
\m{Q\subs\DD}. So let \m{\rho_1\in F^{\inV}} be an \nbdd{m}ary relation and \m{x\in\N^\Gamma} with
\m{x\below\pt{\rho_1}}. Thus, \m{x\leq \pt{\rho_1}} and
\m{\supp{x}=\supp{\pt{\rho_1}} \in\powersetfin{\Gamma}}.
Putting \m{n\defeq \sum_{\gamma\in\supp{x}} x(\gamma)}, we construct a
parameter tuple \m{z\in\Gamma^n} as follows: for each
\m{\gamma\in\supp{x}}, we add \m{x(\gamma)} copies of \m{\gamma}
to~\m{z}, in any selected order. Then we have
\m{\rho_2\defeq \Runary{z}\in\DD}. If \m{\rho_1=\CarrierSet[m]}, then
\m{\CarrierSet\in\Gamma}, \m{\pt{\rho_1}\apply{\CarrierSet}=m} and
\m{\pt{\rho_1}\apply{\gamma}=0} for all
\m{\gamma\in\Gamma\setminus\set{\CarrierSet}}. Now
\m{\supp{x}=\supp{\pt{\rho_1}} = \set{\CarrierSet}}, so
\m{z = \apply{\CarrierSet,\dotsc,\CarrierSet}}, where \m{\CarrierSet}
occurs exactly \nbdd{x\apply{\CarrierSet}}many times.
Then \m{p\apply{\Runary{z}}=z}. On the other hand, if
\m{\rho_1\subsetneq\CarrierSet[m]}, then
\m{\CarrierSet\notin\supp{\pt{\rho_1}}}, so~\m{\CarrierSet} does not
occur in the image of~\m{z}. Therefore,
\m{p\apply{\Runary{z}}=z}. In both cases we have
\m{\pt{\rho_2}=x} since \m{p\apply{\rho_2}=p\apply{\Runary{z}}=z}
and~\m{z} contains the appropriate number of relations~\m{\gamma} for each
\m{\gamma\in\supp{x}}. Hence, our assumption on~\m{x} yields
\m{\pt{\rho_2}=x\below\pt{\rho_1}}, and so
Lemma~\ref{lem:downset-implies-Pol} implies
\m{\set{\rho_1}^{\poL}\subs\set{\rho_2}^{\poL}}. Consequently, we get
\m{\rho_2\in \set{\rho_2}^{\poL\inV}\subs\set{\rho_1}^{\poL\inV}\subs
F^{\inV\poL\inV} = F^{\inV}}, and thus
\m{x=\pt{\rho_2}\in\ptOp\fapply{F^{\inV}}=\psi\apply{F}} as desired.
\par
It is obvious that~\m{\psi} is order preserving. To prove that it is
order reflecting, consider \m{F_1=Q_1^{\poL}} and \m{F_2=Q_2^{\poL}} for
\m{Q_1,Q_2\subs\DD} such that \m{\psi\apply{F_1}\subs\psi\apply{F_2}} and
any \m{\rho_1\in F_1^{\inV}}. Then we have
\m{\pt{\rho_1}\in\psi\apply{F_1}\subs\psi\apply{F_2}=\ptOp\fapply{F_2^{\inV}}},
so there is some \m{\rho_2\in F_2^{\inV}} such that
\m{\pt{\rho_1}=\pt{\rho_2}}.
Applying Lemma~\ref{lem:downset-implies-Pol} twice, we obtain
\m{\set{\rho_1}^{\poL} = \set{\rho_2}^{\poL}} and hence
\m{\rho_1\in\set{\rho_1}^{\poL\inV} =\set{\rho_2}^{\poL\inV}\subs
F_2^{\inV\poL\inV} = F_2^{\inV}}.
Thus \m{F_1^{\inV}\subs F_2^{\inV}} and
\m{F_1 = F_1^{\inV\poL}\supseteq F_2^{\inV\poL} = F_2}.
\par
Order reflection clearly implies that \m{\psi} is injective. It needs
to be shown that \m{\phi\apply{\psi\apply{F}} = F} for every
\m{F=Q^{\poL}} where \m{Q\subs\DD}.
As \m{\psi\apply{F}\in\DS[{\N^\Gamma,\below}]}, we have
\m{\psi\apply{F} = \downset[{\apply{\N^\Gamma,\below}}]{\psi\apply{F}}
=\downset[{\apply{\N^\Gamma,\below}}]{\ptOp\fapply{F^{\inV}}}
=I\apply{F^{\inV}}}, whence it follows that
\m{\phi\apply{\psi\apply{F}}=\phi\apply{I\apply{F^{\inV}}}=F^{\inV\poL}=F}.
\end{proof}

\section{Results}\label{sec:results}
By Corollary~\ref{cor:upper-bound} from the previous section we have
transferred the task of counting the number of clones given by relations
that are disjunctively definable over a fixed (commonly finite) set of unary
predicates~\m{\Gamma} to the investigation of the downsets of the poset
\m{\apply{\N^\Gamma,\below}}. This is a special case of the poset
\m{\apply{\N^I,\below}} from Section~\ref{sect:prelim}, where we now have
the additional assumption that~\m{I} is finite.
Concerning downsets of this poset, we first verify the following
general facts.
\begin{lemma}\label{lem:poset-on-N-to-I}
Let\/ \m{\ordset{\genericOrder}{\N^I}{\below}} be the poset defined in
Section~\ref{sect:prelim}.
\begin{enumerate}[(a)]
\item\label{item:YF-downset}
      The set \m{\YF\defeq \lset{x\in\N^I}{\supp{x} = I}} is a
      downset of\/~\m{\order}, and the induced subposet\/
      \m{\ordset{\YF}{\YF}{\restrordrel[{\below}]{\YF}}} satisfies
      \m{\order[{\YF}] \cong \apply{\apply{\Np}^I, \leq }
                       \cong \apply{\N^I,\leq}}.
\item\label{item:YsubsJ-downset}
      For every \m{J\subs I} the set
      \m{Y_{\subs J}\defeq \lset{x\in\N^I}{\supp{x}\subs J}} is a downset
      of\/~\m{\order}; the induced subposet\/
      \m{\ordset{\Ysubs{J}}{\Ysubs{J}}{\restrordrel[{\below}]{\Ysubs{J}}}}
      is isomorphic to \m{\apply{\N^J,\below}}.
\item\label{item:N-to-I-downset-decomposition}
      We have\/ \m{\N^I=\YF\cup\bigcup_{i\in I} \Ysubs{I\setminus\set{i}}}.
\end{enumerate}
\end{lemma}
\begin{proof}
\begin{enumerate}[(a)]
\item If \m{x\in\YF} and \m{y\in\N^I} satisfies \m{y\below x}, then
      \m{I=\supp{x}=\supp{y}}. Hence, \m{y\in\YF}, and so
      \m{\YF\in\DS}. Moreover, for \m{x\in\N^I}, we have \m{x\in\YF} if and
      only if \m{x\apply{i}\neq 0} for all \m{i\in I}, \ie\ if
      \m{x\in\apply{\Np}^I}. Thus the identical map induces an isomorphism
      \m{\functionhead{\id}{\order[{\YF}]}{\apply{\apply{\Np}^I,\leq}}}
      since all elements \m{x,y\in\apply{\Np}^I} automatically satisfy
      \m{\supp{x}=I=\supp{y}}. Furthermore, the function
      \m{\functionhead{h}{\apply{\apply{\Np}^I,\leq}}{\apply{\N^I,\leq}}}
      given by \m{h\apply{x} = \apply{x\apply{i}-1}_{i\in I}} obviously is
      an isomorphism, too.
\item Consider \m{J\subs I} and \m{x\in\Ysubs{J}}. Any \m{y\below x}
      fulfils \m{\supp{y}=\supp{x}\subs J}, so \m{y\in\Ysubs{J}}. Hence,
      \m{\Ysubs{J}\in\DS}. Define
      \m{\functionhead{\pr_J}{\Ysubs{J}}{\N^J}} by letting
      \m{\pr_{J}\apply{x}\defeq \apply{x\apply{j}}_{j\in J}}; conversely,
      for every element \m{y\in\N^J} define \m{\emb_{J}\apply{y}} by
      \m{\emb_{J}\apply{y}\apply{i}\defeq y\apply{i}} if \m{i\in J} and
      \m{\emb_{J}\apply{y}\apply{i}\defeq 0}, otherwise.
      The map \m{\functionhead{\emb_{J}}{\N^J}{\Ysubs{J}}} is
      inverse to \m{\pr_{J}} as all \m{x\in \Ysubs{J}} satisfy
      \m{\supp{x}\subs J}, \ie\ every \m{i\in I\setminus J} does not belong
      to \m{\supp{x}} and thus \m{x\apply{i}= 0}. For the same
      reason we also have \m{\supp{x}=\supp{\pr_{J}\apply{x}}} for all
      \m{x\in\Ysubs{J}}, whence \m{\pr_{J}} is a homomorphism. Besides,
      the equality \m{\supp{y}=\supp{\emb_{J}\apply{y}}} holds for all
      \m{y\in\N^J}, so \m{\emb_{J}} is a homomorphism, too.
\item The inclusion
      \m{\N^I\supseteq\YF\cup\bigcup_{i\in I} \Ysubs{I\setminus\set{i}}} holds
      by definition. Let moreover \m{x\in \N^J\setminus\YF}, then
      \m{\supp{x}\subsetneq I}, so there exists some \m{i\in
      I\setminus\supp{x}}, \ie\ \m{\supp{x}\subs I\setminus\set{i}}.
      This proves that \m{x\in\bigcup_{i\in I}\Ysubs{I\setminus\set{i}}}.
      \qedhere
\end{enumerate}
\end{proof}
\par

\begin{theorem}\label{thm:downsets-of-poset-on-N-to-I}
We have\/ \m{\abs{\DS[{\N^I,\below}]}=\aleph_{0}}
for all finite \m{I\neq \emptyset}.
\end{theorem}
\begin{proof}
Since \m{\mathord{\below}\subs\mathord{\leq}}, we clearly have
\m{\DS[{\N^I,\leq}]\subs\DS[{\N^I,\below}]}, which together with
Corollary~\ref{cor:downsets-of-N-to-I} proves that
\m{\abs{\DS[{\N^I,\below}]}\geq
   \abs{\DS[{\apply{\N,\leq}^I}]}=\aleph_{0}} for all finite
non\dash{}empty \m{I}. We shall prove by induction on \m{\abs{I}} that
\m{\abs{\DS[{\N^I,\below}]}\leq\aleph_{0}} holds for all finite
sets \m{I}. The basis is the case \m{\abs{I}=0}, \ie\ \m{I=\emptyset}. Then
\m{\abs{\N^I} = 1}, so we are dealing with a finite poset having only
finitely many downsets. Now assume \m{\abs{I}>0}, \ie\ \m{I\neq\emptyset},
and suppose we know the truth of the claim already for all finite \m{J},
\m{\abs{J}<\abs{I}}. In particular, we have the induction hypothesis for
all \m{I\setminus\set{i}} where \m{i\in I}. We define
\begin{equation*}
\function{\delta}{\DS[{\N^I,\below}]}{%
     \DS[{\order[\YF]}]\times \prod_{i\in
     I}\DS[{\order[{\Ysubs{I\setminus\set{i}}}]}]}%
     {X}{\apply{X\cap\YF,\apply{X\cap\Ysubs{I\setminus\set{i}}}_{i\in I}}.}
\end{equation*}
By Lemma~\ref{lem:poset-on-N-to-I}\eqref{item:YF-downset},
                                  \eqref{item:YsubsJ-downset}
in combination with Lemma~\ref{lem:ds-to-subposets}, this map is
well\dash{}defined. Moreover, due to
Lemma~\ref{lem:poset-on-N-to-I}\eqref{item:N-to-I-downset-decomposition},
we have
\begin{equation*}
\apply{X\cap\YF}\cup
   \bigcup_{i\in I}\apply{X\cap\Ysubs{I\setminus\set{i}}}
   =X\cap\apply{\YF\cup\bigcup_{i\in I}\Ysubs{I\setminus\set{i}}}
   =X\cap\N^I = X,
\end{equation*}
so~\m{\delta} is injective. Hence,
\m{\abs{\DS[{\N^I,\below}]}\leq
   \abs{\DS[{\order[{\YF}]}]\times
   \prod_{i\in I}\DS[{\order[{\Ysubs{I\setminus\set{i}}}]}]}}.
Since for \m{i\in I}, we have
\m{\order[{\Ysubs{I\setminus\set{i}}}]
   \cong\apply{\N^{I\setminus\set{i}},\below}}
by Lemma~\ref{lem:poset-on-N-to-I}\eqref{item:YsubsJ-downset},
Lemma~\ref{lem:iso-bij-DS} together with the induction hypothesis yields
\m{\abs{\DS[{\order[{\Ysubs{I\setminus\set{i}}}]}]}
  =\abs{\DS[{\N^{I\setminus\set{i}},\below}]}\leq\aleph_{0}}.
Similarly, \m{\order[\YF]\cong\apply{\N,\leq}^I}
by Lemma~\ref{lem:poset-on-N-to-I}\eqref{item:YF-downset}, so
Lemma~\ref{lem:iso-bij-DS} and Corollary~\ref{cor:downsets-of-N-to-I}
jointly imply
\m{\abs{\DS[{\order[{\YF}]}]}=\abs{\DS[{\apply{\N,\leq}^I}]}=\aleph_{0}}.
Consequently, as a finite product of countable sets, one of
which is infinite, the co\dash{}domain of~\m{\delta} has
cardinality~\m{\aleph_{0}}, whence \m{\abs{\DS[{\N^I,\below}]}} is
countable.
\end{proof}
\par
As a consequence the number of clones on a fixed set determined by
disjunctions of finitely many unary predicates is countable.
\begin{corollary}\label{cor:clones-at-most-countable}
For every finite unary relational language~\m{\Gamma} we have
\[\abs{\lset{\Pol{Q}}{Q\subs\DD}} \leq \aleph_0.\]
\end{corollary}
\begin{proof}
From Corollary~\ref{cor:upper-bound} we know
\m{\abs{\lset{\Pol{Q}}{Q\subs\DD}} \leq
\abs{\DS[{\N^\Gamma,\below}]}}.
By Theorem~\ref{thm:downsets-of-poset-on-N-to-I} the latter is
countably infinite if~\m{\Gamma\neq \emptyset}, and it has two elements
if \m{\Gamma=\emptyset} (in this case we are dealing only with the
clone of all operations).
\end{proof}

This already demonstrates that a classification of such clones (as
requested
in~\cite[Section~6]{CreignouHermannKrokhinSalzerComplexityOfClausalConstraintsOverChains})
is not a hopeless task.
\par

As a final step we wish to show that in all relevant cases, our
cardinality bound from Corollary~\ref{cor:clones-at-most-countable} is
tight. This generalizes an argument given
in~\cite[Proposition~3.1]{VargasCRelsCclones} regarding the number of
clones determined by (mixed) clausal relations (an important subcase of
the relations considered
in~\cite{CreignouHermannKrokhinSalzerComplexityOfClausalConstraintsOverChains}).

\begin{proposition}\label{prop:infinitely-many-dd-clones}
If\/~\m{\Gamma} is any unary relational language with
carrier set~\m{\CarrierSet} containing a non\dash{}trivial basic unary
relation~\m{\gamma\in\Gamma} such that
\m{\emptyset\neq \gamma\subsetneq \CarrierSet}, then the lattice
\m{\apply{\lset{\Pol{Q}}{Q\subs\DD},\mathord{\subs}}} contains a strictly
descending \nbdd{\omega}chain of finitely related clones over~\m{\DD}.
\end{proposition}
\begin{proof}
The proof of this fact is constructive. Fix \m{\gamma\in\Gamma} with
the properties claimed in the proposition. For every \m{m\in\Np} we let
\m{\rho_m\defeq \Runary{\gamma,\dotsc,\gamma}}, where the
parameter~\m{\gamma} occurs exactly \m{m}~times.
By Lemma~\ref{lem:duplicating-parameters}, we have
\m{\Pol{\set{\rho_{m}}}\subs\Pol{\set{\rho_{m-1}}}} for all \m{m\in\N},
\m{m\geq 2}.
We only have to prove that these inclusions are strict. This will be done
by exhibiting an \nbdd{m}ary function \m{f\in\Pol{\set{\rho_{m-1}}}} that
does not preserve~\m{\rho_m}.
\par
Since \m{\gamma\subsetneq\CarrierSet}, there is an element
\m{0\in \CarrierSet} such that \m{0\notin \gamma}. Moreover, as
\m{\gamma} is not empty, there is some element
\m{1\in \CarrierSet\setminus\set{0}} such that \m{1\in\gamma}. We define
\m{f\apply{\liste{x}{m}} = 0} if at least \m{m-1} entries
in~\m{\apply{\liste{x}{m}}} are equal to~\m{0} and we put
\m{f\apply{\liste{x}{m}} = 1} everywhere else.
\par
Let \m{e_j\in\CarrierSet[m]} be the tuple whose \nbdd{j}th entry is~\m{1}
and which is~\m{0} otherwise. Since \m{1\in\gamma}, we have
\m{\liste{e}{m}\in \rho_m}. However, the definition of~\m{f} gives us
\m{\composition{f}{\liste{e}{m}} = \mathbf{0}}, the tuple containing only
zeros. As \m{0\notin \gamma}, we have \m{\mathbf{0}\notin\rho_m} and
hence \m{f\notin\Pol{\set{\rho_m}}}.
\par
On the other hand, if \m{r_1,\dotsc,r_m\in \rho_{m-1}} and we consider
these tuples as columns of an \nbdd{((m-1)\times m)}matrix, then each
column contains a non\dash{}zero entry (because
\m{\mathbf{0}\notin\rho_{m-1}}). As \m{m>m-1}, by the pigeonhole
principle there must be one row~\m{x} of the matrix that contains at
least two entries distinct from zero. In other words, there cannot be
more than \m{m-2} zero entries in~\m{x}. Now applying~\m{f} to~\m{x}
yields \m{f(x) = 1\in\gamma}.
Therefore, \m{\composition{f}{\liste{r}{m}}\in\rho_{m-1}}.
\end{proof}

Combining Corollary~\ref{cor:clones-at-most-countable} and
Proposition~\ref{prop:infinitely-many-dd-clones} we can pinpoint the
exact number of clones determined by disjunctions of
non\dash{}trivial unary relations from a finite parameter
set~\m{\Gamma}. This answers, in particular, the
question regarding the number of clausal clones on finite sets that was
stated to be open in~\cite{Edith-thesis,%
                           BehVarUniqueInclusionsOfMaxCclones2018}.
\begin{corollary}\label{cor:tightness}
If\/~\m{\Gamma} is a finite unary relational language with carrier
set~\m{\CarrierSet} containing a non\dash{}trivial basic unary
relation~\m{\gamma\in\Gamma} such that
\m{\emptyset\neq \gamma\subsetneq \CarrierSet}, then we have\/
\m{\abs{\lset{\Pol{Q}}{Q\subs\DD \text{ finite}}}=
   \abs{\lset{\Pol{Q}}{Q\subs\DD}}=\aleph_0}.
\end{corollary}

\begin{remark}\label{rem:countable-Gamma}
When focussing only on finitely related polymorphism clones, we can
extend the scope of our arguments a little bit. That is to say, if we
restrict the encoding map~\m{I}
(cf.\ Proposition~\ref{prop:encoding-implies-Pol})
to~\m{\powersetfin{\DD}} and only consider clones \m{\Pol{Q}} given by a
finite subset \m{Q\subs\DD}, we can still obtain the upper
bound~\m{\aleph_0} on their cardinality when~\m{\Gamma} is countably
infinite. Namely, if \m{Q =\set{\liste{\rho}{N}}} for some \m{N\in\N},
then \m{I\apply{Q} =
\bigcup_{i=1}^N \downset[{\apply{\N^\Gamma,\below}}]{\pt{\rho_i}}} is a
finitely generated downset, where for each \m{i\in\set{1,\dotsc,N}} the
tuples in the principal downset
\m{\downset[{\apply{\N^\Gamma,\below}}]{\pt{\rho_i}}} have a fixed
finite support \m{J_i\subs\Gamma}. Hence, all tuples in \m{I\apply{Q}}
have their support within the finite set
\m{J\apply{Q} = \bigcup_{i=1}^N J_i\subs\Gamma}. By projecting the
patterns to these indices (as in the proof of
Lemma~\ref{lem:poset-on-N-to-I} this does not change the support of the
tuples), we obtain a different encoding
\m{K(Q) = \bigcup_{i=1}^N
\downset[{\apply{\N^{J\apply{Q}},\below}}]{\pr_{J\apply{Q}}\pt{\rho_i}}},
which is a downset of \m{\apply{\N^{J\apply{Q}},\below}}. Therefore,
\m{K\apply{Q} \in \bigcup_{J\in\powersetfin{\Gamma}} \DS[{\N^J,\below}]}.
Since \m{\Gamma} is countable, \m{\powersetfin{\Gamma}} is countable and
the codomain of~\m{K} is a countable union of countably infinite sets
(see Theorem~\ref{thm:downsets-of-poset-on-N-to-I}),
hence again countably infinite. With analogous arguments as in
Proposition~\ref{prop:encoding-implies-Pol} through Corollary~\ref{cor:upper-bound} we can thus show that
\m{\abs{\lset{\Pol{Q}}{Q\subs \DD\text{ finite}}} \leq \abs{\im K} \leq
\aleph_0}.
Regarding the analogue of Proposition~\ref{prop:encoding-implies-Pol}
we note that \m{\emptyset\neq K\apply{Q_1}\subs K\apply{Q_2}} implies
\m{J\apply{Q_1}= J\apply{Q_2}} and so
\m{\pr_{J\apply{Q_1}}\pt{\rho}\below \pr_{J\apply{Q_2}}\pt{\rho'}}
yields \m{\pt{\rho}\below\pt{\rho'}} for \m{\rho\in Q_1} and any
\m{\rho'\in Q_2}.
As \m{\supp{\pt{\rho}} = \supp{\pt{\rho'}} \subs J\apply{Q_2}} is
finite, Corollary~\ref{cor:duplicating-permuting-parameters}
and Lemma~\ref{lem:downset-implies-Pol} are applicable to derive the
conclusion of Proposition~\ref{prop:encoding-implies-Pol}.
\par
%
%If \m{K(Q_1)\subs K(Q_2)} and \m{\rho\in Q_1}, then
%\m{\supp{\pt{\rho}}\subs J(Q_1)} and we have
%\m{\pr_{J(Q_1)}\pt{\rho}\in \downset{\pr_{J(Q_1)}\pt{\rho} }\subs
%K(Q_1)\subs K(Q_2) =\bigcup_{\rho'\in Q_2}
%\downset{\pr_{J(Q_2)}\pt{\rho'}}}. Note that \m{K(Q_1) \subs \N^{J(Q_1)}}
%and \m{K(Q_2)\subs \N^{J(Q_2)}} whence the inclusion
%\m{\emptyset\subsetneq K(Q_1)\subs K(Q_2)} is only possible if
%\m{J(Q_1) =J(Q_2)\eqdef J} and \m{\apply{\N^{J(Q_1)},\below} = \order =
%\apply{\N^{J(Q_2)},\below}} form a common poset. It follows that
%\m{\pr_J \pt{\rho} \below \pr_J\pt{\rho'}}, and this means
%\m{\supp{\pt{\rho}} = \supp{\pr_J\pt{\rho}} = \supp{\pr_J\pt{\rho'}}
%=\supp{\pt{\rho'}}}. Consequently,
%\m{\pr_J{\pt{\rho}} \leq \pr_J{\pt{\rho'}}}
%implies \m{\pt{\rho}\leq \pt{\rho'}} and hence we also have
%\m{\pt{\rho}\below\pt{\rho'}}. This gives via
%Lemma~\ref{lem:downset-implies-Pol} that
%\m{\Pol{Q_2}\subs\Pol{\set{\rho'}}\subs\Pol{\set{\rho}}}
%for the arbitrarily chosen \m{\rho\in Q_1}.
%\par
%
Tightness of the cardinality bound is again provided by
Proposition~\ref{prop:infinitely-many-dd-clones}, which does not require
finiteness of~\m{\Gamma}.
\end{remark}

Combining these explanations with Corollary~\ref{cor:tightness} we can
strengthen our result as follows:
\begin{corollary}\label{cor:tightness-improved}
If\/~\m{\Gamma} is a countable unary relational language
on the carrier
set~\m{\CarrierSet} containing a non\dash{}trivial basic unary
relation~\m{\gamma\in\Gamma} such that
\m{\emptyset\neq \gamma\subsetneq \CarrierSet}, then we have\/
\m{\abs{\lset{\Pol{Q}}{Q\subs\DD \text{ finite}}}=\aleph_0}.
\end{corollary}

\section*{Acknowledgements}
The authors are grateful to one of the anonymous referees suggesting the
idea for the alternative proof of Corollary~\ref{cor:upper-bound} that
is presented in Proposition~\ref{prop:upper-bound-by-injection}.

%\bibliographystyle{amsalpha}
%\bibliography{bibtexdatabase}
\newcommand{\etalchar}[1]{$^{#1}$}
\def\ocirc#1{\ifmmode\setbox0=\hbox{$#1$}\dimen0=\ht0 \advance\dimen0
  by1pt\rlap{\hbox to\wd0{\hss\raise\dimen0
  \hbox{\hskip.2em$\scriptscriptstyle\circ$}\hss}}#1\else {\accent"17 #1}\fi}
  \def\Palatalization#1{\bgroup\fontencoding{T1}\selectfont\v{#1}\egroup}
  \def\Dj{\bgroup\fontencoding{T1}\selectfont\DJ\egroup}
  \providecommand*{\doi}[1]{\href{http://dx.doi.org/\detokenize{#1}}{\detokenize{#1}}}
  \ifx\SetBibliographyCyrillicFontfamily\undefined\def\SetBibliographyCyrillicFontfamily{\relax}\fi
  \def\Cyr#1{\bgroup\SetBibliographyCyrillicFontfamily\fontencoding{T2A}\selectfont{#1}\egroup}
  \ifx\AvailableAt\undefined\def\AvailableAt{available at}\fi
  \ifx\AVAILABLEat\undefined\def\AVAILABLEat{Available at}\fi
  \ifx\OnlineAvblAt\undefined\def\OnlineAvblAt{online}\fi
  \ifx\PreprintAvblAt\undefined\def\PreprintAvblAt{pre-print}\fi
\providecommand{\bysame}{\leavevmode\hbox to3em{\hrulefill}\thinspace}
\providecommand{\MR}{\relax\ifhmode\unskip\space\fi MR }
\providecommand{\MRhref}[2]{%
  \href{http://www.ams.org/mathscinet-getitem?mr=#1}{#2}
}
\providecommand{\href}[2]{#2}

\end{document}